\documentclass[11pt]{amsart}
\usepackage{amssymb}
\usepackage{amsmath}
\usepackage{amsthm}
\usepackage{latexsym}
\usepackage{amsfonts}
\usepackage{graphicx}
\usepackage{graphics}
\usepackage[T1]{pbsi}

\newtheorem{lem}{Lemma}
\newtheorem{cor}{Corollary}
\newtheorem{Def}{Definition}

\theoremstyle{definition}
\newtheorem{rem}{Remark}

\newcommand{\dis}{\displaystyle}
\newcommand{\el}{\ell}

\newcommand{\ra}{\;\rightarrow\;}

\newcommand{\Ga} {{\varGamma}}

\newcommand{\OO} {{\varOmega}}
\newcommand{\De} {{\varDelta}}
\newcommand{\e}{\varepsilon }

\newcommand{\la}{\lambda }
\newcommand{\mi}{\mu }

\newcommand{\si}{\sigma }

\newcommand{\R}{\mathbb{R}}

\newcommand{\ct}{{\mathcal{T}}}

\newcommand{\cm}{{\mathcal{M}}}

\newcommand{\ld}{\ldots}

\newcommand{\hs}{\hfill$\square$}

\newcommand{\bbb}[1]{\mbox{\boldmath$#1$}}

\begin{document}

\title{Dyadic weights on $\bbb{\R^n}$ and Reverse H\"{o}lder inequalities}
\author{Eleftherios N. Nikolidakis, Antonios D. Melas}
\footnotetext{\hspace{-0.5cm}This research has been co-financed by the European Union and Greek national funds through the Operational Program "Education and Lifelong Learning" of the National Strategic Reference Framework (NSRF), aristeia code: MAXBELLMAN 2760, research code:70/3/11913.} 
\date{}
\maketitle
\noindent
{\bf Abstract:} We prove that for any weight $\phi$ defined on $[0,1]^n$ that satisfies a reverse H\"{o}lder inequality with exponent $p>1$ and constant $c\ge1$ upon all dyadic subcubes of $[0,1]^n$, it's non increasing rearrangement $\phi^\ast$, satisfies a reverse H\"{o}lder inequality with the same exponent and constant not more than $2^nc-2^n+1$, upon all subintervals of $[0,1]$ of the form $[0,t]$, $0<t\le1$. This gives as a consequence, according to the results in \cite{8}, an interval $[p,p_0(p,c))=I_{p,c}$, such that for any $q\in I_{p,c}$, we have that $\phi\in L^p$.\bigskip\\
\noindent
%
%
\section{Introduction}\label{sec1}
\noindent

The theory of Muckenhoupt's weights has been proved to be an important tool in analysis. One of the most important facts about these is their self improving property. A way to express this is through the so called reverse H\"{o}lder inequalities (see \cite{2}, \cite{3} and \cite{7}).

Here we will study such inequalities on a dyadic setting. We will say that the measurable function $g:[0,1]\ra\R^+$ satisfies the reverse H\"{o}lder inequality with exponent $p>1$ and constant $c\ge1$ if the inequality
\begin{eqnarray}
\frac{1}{b-a}\int^b_ag^p(u)du\le c\bigg(\frac{1}{b-a}\int^b_ag(u)du\bigg)^p, \label{eq1.1}
\end{eqnarray}
holds for every subinterval of $[0,1]$.

In \cite{1} it is proved the following\vspace*{0.2cm}\\
\noindent
{\bf Theorem A.} {\em Let $g$ be a non-increasing function defined on $[0,1]$, which satisfies (\ref{eq1.1}) on every interval $[a,b]\subseteq[0,1]$. Then if we define $p_0>p$ as the root of the equation
\begin{eqnarray}
\frac{p_0-p}{p_0}\bigg(\frac{p_0}{p_0-1}\bigg)^p\cdot c=1,  \label{eq1.2}
\end{eqnarray}
we have that $g\in L^q([0,1])$, for any $q\in[p,p_0)$. Additionally $g$ satisfies for every $q$ in the above range a reverse H\"{o}lder inequality for possibly another real constant $c'$. Moreover the result is sharp, that is the value $p_0$ cannot be increased.}

Now in \cite{4} or \cite{5} it is proved the following \vspace*{0.2cm} \\
\noindent
{\bf Theorem B.} {\em If $\phi:[0,1]\ra\R^+$ is measurable satisfying (\ref{eq1.1}) for every $[a,b]\subseteq[0,1]$, then it's non-increasing rearrangement $\phi^\ast$, satisfies the same inequality with the same constant $c$.}

Here by $\phi^\ast$ we denote the non-increasing rearrangement of $\phi$, which is defined on $(0,1]$ by
\[
\phi^\ast(t)=\sup_{E\subseteq[0,1]\atop\mid E\mid=t}\bigg\{\inf_{x\in E}\mid\phi(x)\mid\bigg\}, \qquad t\in(0,1].
\]
This can be defined also as the unique left continuous, non-increasing function, equimeasurable to $|\phi|$, that is, for every $\la>0$ the following equality holds:
\[
\mid\{\phi>\la\}\mid=\mid\{\phi^\ast>\la\}\mid,
\]
where by $\mid\cdot\mid$ we mean the Lesbesgue measure on $[0,1]$.

An immediate consequence of Theorem B, is that Theorem A can be generalized by ignoring the assumption of the monotonicity of the function $g$.

Recently in \cite{8} it is proved the following \vspace*{0.2cm} \\
\noindent
{\bf Theorem C.} {\em Let $g:(0,1]\ra\R^+$ be non-increasing which satisfies (\ref{eq1.1}) on every interval of the form $(0,t]$, $0<t\le1$. That is the following holds
\begin{eqnarray}
\frac{1}{t}\int^t_0g^p(u)du\le c\cdot\bigg(\frac{1}{t}\int^t_0g(u)du\bigg)^p  \label{eq1.3}
\end{eqnarray}
for every $t\in(0,1]$. Then if we define $p_0$ by (\ref{eq1.2}), we have that for any $q\in[p,p_0)$ the following inequality is true
\begin{eqnarray}
\frac{1}{t}\int^t_0g^q(u)du\le c'\bigg(\frac{1}{t}\int^t_0g(u)du\bigg)^q,  \label{eq1.4}
\end{eqnarray}
for every $t\in(0,1]$ and some constant $c'\ge c$. Thus $g\in L^q((0,1])$ for any such $q$. Moreover the result is sharp, that is we cannot increase $p_0$.}

A consequence of Theorem C is that under the assumption that $g$ is non-increasing, the hypothesis that (\ref{eq1.1}) is satisfied only on the intervals of the form $(0,t]$ is enough for one to realize the existence of a $p'>p$ fir which $g\in L^{p'}([0,1])$.

In several dimensions, as far as we know, there does not exists any similar result as Theorems A, B and C. All we know is the following, which can be seen in \cite{3}. \vspace*{0.2cm} \\
\noindent
{\bf Theorem D.} {\em Let $Q_0\subseteq\R^n$ be a cube and $\phi:Q_0\ra\R^+$ measurable that satisfies
\begin{eqnarray}
\frac{1}{\mid Q\mid}\int_Q\phi^p\le c\cdot\bigg(\frac{1}{\mid Q\mid}\int_Q\phi\bigg)^p  \label{eq1.5}
\end{eqnarray}
for fixed constants $p>1$ and $c\ge1$ and every cube $Q\subseteq Q_0$. Then there exists $\e=\e(n,p,c)$ such that the following inequality holds;
\begin{eqnarray}
\frac{1}{\mid Q\mid}\int_Q\phi^q\le c'\bigg(\frac{1}{\mid Q\mid}\int_Q\phi\bigg)^q  \label{eq1.6}
\end{eqnarray}
for every $q\in[p,p+\e)$, any cube $Q\subseteq Q_0$ and some constant $c'=c'(q,p,n,c)$.}

In several dimensions no estimate of the quantity $\e$, has been found. The purpose of this work is to study the multidimensional case in the dyadic setting. More precisely we consider a measurable function $\phi$, defined on $[0,1]^n=Q_0$, which satisfies (\ref{eq1.5}) for any $Q$, dyadic subcube of $Q_0$. These cubes can be realized by bisecting the sides of $Q_0$, then bisecting it's side of a resulting dyadic cube and so on. We define by $\ct_{2^n}$ the respective tree consisting of those mentioned dyadic subcubes of $[0,1]^n$. Then we will prove the following: \vspace*{0.2cm} \\
\noindent
{\bf Theorem 1.} {\em Let $\phi:Q_0=[0,1]^n\ra\R^+$ be such that
\begin{eqnarray}
\frac{1}{\mid Q\mid}\int_Q\phi^p\le c\cdot\bigg(\frac{1}{\mid Q\mid}\int_Q\phi\bigg)^p,  \label{eq1.7}
\end{eqnarray}
for any $Q\in\ct_{2^n}$ and some fixed constants $p>1$ and $c\ge1$. Then, if we set $h=\phi^\ast$ the non-increasing rearrangement of $\phi$, the following inequality is true
\begin{eqnarray}
\hspace*{1.5cm}\frac{1}{t}\int_0^th^p(u)du\le(2^nc-2^n+1)\bigg(\frac{1}{t}\int^t_0h(u)du\bigg)^p, \quad \text{for any} \quad t\in [0,1].  \label{eq1.8}
\end{eqnarray}}
As a consequence $h=\phi^\ast$ satisfies the assumptions of Theorem C, which can be applied and produce an $\e_1=\e_1(n,p,c)>0$ such that $h$ belongs to $L^q([0,1])$ for any $q\in[p,p+\e_1)$. Thus $\phi\in L^q([0,1]^n)$ for any such $q$. That is we can find an explicit value of $\e_1$. This is stated as Corollary 3.1 and is presented in the last section of this paper.

As a matter of fact we prove Theorem 1 in a much more general setting. More precisely we consider a non-atomic probability space $(X,\mi)$ equipped with a tree $\ct_k$, that is a $k$-homogeneous tree for a fixed integer $k>1$, which plays the role of dyadic sets as in $[0,1]^n$ (see the definition of Section \ref{sec2}).

As we shall see later, Theorem 1 is independent of the shape of the dyadic sets and depends only on the homogeneity of the tree $\ct_k$. Additionally we need to mention that the inequality (\ref{eq1.8}) cannot necessarily be satisfied, under the assumptions of Theorem 1, of one replaces the intervals $(0,t]$ by $(t,1]$. That is  $\phi^\ast$ is not necessarily a weight on $(0,1]$ satisfying a reverse H\"{o}lder inequality upon all subintervals of $[0,1]$ (see \cite{5}).

Additionally we mention that in \cite{6} the study of the dyadic $A_1$-weights appears, where one can find for any $c>1$ the best possible range $[1,p)$, for which the following holds: $\phi\in A^d_1(c)\Rightarrow\phi\in L^q$, for any $q\in[1,p)$. All last results that are connected with $A_1$ dyadic weights $\phi$ and the behavior of $\phi^\ast$ as an $A_1$-weight on $\R$, can be seen in \cite{9}.
\section{Preliminaries}\label{sec2}
\noindent

Let $(X,\mi)$ be a non-atomic probability space. We give the notion  of a $k$-homogeneous tree on $X$.
\begin{Def}\label{Def2.1}
Let $k$ be an integer such that $k>1$. A set $\ct_k$ will be called a $k$-homogeneous tree on $X$ if the following hold

(i) $X\in\ct_k$

(ii) For every $I\in\ct_k$, there corresponds a subset $C(I)\subseteq\ct_k$ consisting of $k$ subsets of $I$ such that
\begin{enumerate}
\item[(a)] the elements of $C(I)$ are pairwise disjoint
\item[(b)] $I=\dis\bigcup C(I)$
\item[(c)] $\mi(J)=\dfrac{1}{k}\mi(I)$, for every $J\in C(I)$.
\end{enumerate}
\end{Def}

For example one can consider $X=[0,1]^n$, the unit cube of $\R^n$. Define as $\mi$ the Lebesque measure on this cube. Then the set $\ct_k$ of all dyadic subcubes of $X$ is a tree of homogeneity $k=2^n$, with $C(Q)$ being the set of $2^n$-subcubes of $Q$, obtained by bisecting it's sides, for every $Q\in\ct_k$, starting from $Q=X$.

Let now $(X,\mi)$ be as above and a tree $\ct_k$ on $X$ as in Definition \ref{Def2.1}. From now on, we fix $k$ and write $\ct=\ct_k$. For any $I\in\ct$, $I\neq X$ we set $I^\ast$ the smallest element of $\ct$ such that $I^\ast\supsetneq I$. That is $I^\ast$ is the unique element of $\ct$ such that $I\in C(I^\ast)$. We call $I^\ast$ the father of $I$ in $\ct$. Then $\mi(I^\ast)=k\mi(I)$.
\begin{Def}\label{Def2.2}
For any $(X,\mi)$ and $\ct$ as above we define the dyadic maximal operator on $X$ with respect to $\ct$, noted as $\cm_\ct$, by
\setcounter{equation}{0}
\begin{eqnarray}
\cm_\ct\phi(X)=\sup\bigg\{\frac{1}{\mi(I)}\int_I\mid\phi\mid d\mi:x\in I\in\ct\bigg\},  \label{eq2.1}
\end{eqnarray}
for any $\phi\in L^1(X,\mi)$.
\end{Def}
\setcounter{rem}{0}
\begin{rem}\label{rem2.1}
It is not difficult to see that the maximal operator defined by (\ref{eq2.1}) satisfies a weak-type (1,1) inequality, which is the following:
\[
\mi(\{\cm_\ct\phi>\la\})\le\frac{1}{\la}\int_{\{\cm_\ct\phi>\la\}}\phi d\mi, \qquad \la>0.
\]
It is not difficult to see that the above inequality is best possible for every $\la>0$, and is responsible for the fact that $\ct$ differentiates $L^1(X,\mi)$, that is the following holds: $\dis\lim_{x\in I\in\ct\atop\mi(I)\ra0}\dfrac{1}{\mi(I)}\int_I\phi d\mi=\phi(x)$, $\mi$-almost everywhere on $X$.
This can be seen in \cite{4}.
\end{rem}

We will also need the following lemma which can be also seen in \cite{4}.
\begin{lem}\label{lem2.1}
Let $\phi$ be non-negative function defined on $E\cup\widehat{E}\subseteq X$ such that
\begin{eqnarray}
\frac{1}{\mi(E)}\int_E\phi d\mi=\frac{1}{\mi(\widehat{E})}\int_{\widehat{E}}\phi d\mi\equiv A,  \label{eq2.2}
\end{eqnarray}
Additionally suppose that
\begin{eqnarray}
\phi(x)\le A, \qquad \text{for every} \qquad x\notin E\cap\widehat{E}, \label{eq2.3}
\end{eqnarray}
and
\begin{eqnarray}
\phi(x)\le\phi(y), \qquad \text{for every}\qquad X\in\widehat{E}\setminus E,  \qquad{and} \qquad y\in E, \label{eq2.4}
\end{eqnarray}
Then, for every $p>1$ the following inequality holds
\begin{eqnarray}
\frac{1}{\mi(E)}\int_E\phi^pd\mi\le\frac{1}{\mi(\widehat{E})}\int_{\widehat{E}}\phi^pd\mi, \label{eq2.5}
\end{eqnarray}
\end{lem}
\section{Weights on $(X,\mi,\ct)$}\label{sec3}
\noindent

We proceed now to the \vspace*{0.2cm} \\
\noindent
{\bf Proof of Theorem 1}. We suppose that $\phi$ is non-negative defined on $(X,\mi)$ and satisfies a reverse H\"{o}lder inequality of the form
\setcounter{equation}{0}
\begin{eqnarray}
\frac{1}{\mi(I)}\int_I\phi^pd\mi\le c\cdot\bigg(\frac{1}{\mi(I)}\int_I\phi d\mi\bigg)^p,  \label{eq3.1}
\end{eqnarray}
for every $I\in\ct$, where $c,p$ are fixed such that $p>1$ and $c\ge1$. We will prove that for any $t\in(0,1]$ we have that
\begin{eqnarray}
\frac{1}{t}\int^t_0[\phi^\ast(u)]^pdu\le(kc-k+1)\bigg(\frac{1}{t}\int^t_0\phi^\ast(u)du\bigg)^p,
\label{eq3.2}
\end{eqnarray}
where $\phi^\ast$ is the non-increasing rearrangement of $\phi$, defined as in Remark \ref{rem1.1}, on $(0,1]$, and $k$ is the homogeneity of $\ct$. Fix a $t\in(0,1]$ and set
\[
A=A_t=\frac{1}{t}\int^t_0\phi^\ast(u)du.
\]
Consider now the following subset of $X$ defined by
\begin{eqnarray}
E_t=\{x\in X:\cm_\ct\phi(x)>A\},  \label{eq3.3}
\end{eqnarray}
Then for any $x\in E_t$, there exists an element of $\ct$, say $I_x$, such that
\begin{eqnarray}
x\in I_x \qquad \text{and} \qquad \frac{1}{\mi(I_x)}\int_{I_x}\phi d\mi>A.  \label{eq3.4}
\end{eqnarray}
For any such $I_x$ we obviously have that $I_x\subseteq E_t$. We set $S_{\phi,t}=\{I_x:x\in E_t\}$. This is a family of elements of $\ct$ such that $U\{I:I\in S_{\phi,t}\}=E_t$. Consider now those $I\in S_{\phi,t}$ that are maximal with respect to the relation of $\subseteq$. We write this subfamily of $S_{\phi,t}$ as $S'_{\phi,t}=\{I_j:j=1,2,\ld\}$ which is possibly finite. Then $S'_{\phi,t}$ is a disjoint family of elements of $\ct$, because of the maximality of every $I_j$ and the tree structure of $\ct$. (see Definition \ref{Def2.1}).

Then by construction, this family still covers $E_t$, that is $E_t=\dis\bigcup^\infty_{j=1}I_j$. For any $I_j\in S'_{\phi,t}$ we have that $I_j\neq X$, because if $I_j=X$ for some $j$, we could have from (\ref{eq3.4}) that
\[
\int^1_0\phi^\ast(u)du=\int_X\phi d\mi=\frac{1}{\mi(I_j)}\int_{I_j}\phi d\mi>A=\frac{1}{t}\int^t_0\phi^\ast(u)du,
\]
which is impossible, since $\phi^\ast$ is non-increasing on $(0,1]$. Thus, for every $I_j\in S'_{\phi,t}$ we have that $I^\ast_j$ is well defined, but may be common for any two or more elements of $S'_{\phi,t}$. We may also have that $I^\ast_j\subseteq I^\ast_i$ for some $I_j,I_i\in S'_{\phi,t}$.

We consider now the family
\[
L_{\phi,t}=\{I^\ast_j:j=1,2,\ld\}\subseteq\ct.
\]
As we mentioned above, this is not necessarily a pairwise disjoint family. We choose a pairwise disjoint subcollection, by considering those $I^\ast_j$ that are maximal, with respect to the relation $\subseteq$.

We denote this family as
\[
L'_{\phi,t}=\{I^\ast_{j_s}:s=1,2,\ld\}.
\]
Then of course
\[
\bigcup{J:J\in L_{\phi,t}}=\bigcup{J:J\in L'_{\phi,t}}.
\]
Since, each $I_j\in S'_{\phi,t}$ is maximal we should have that
\begin{eqnarray}
\frac{1}{\mi(I^\ast_{j_s})}\int_{I^\ast_{j_s}}\phi d\mi\le A,  \label{eq3.5}
\end{eqnarray}
Now note that every $I^\ast_{j_s}$ contains at least one element of $S'_{\phi,t}$, such that $I\in C(I^\ast_{j_s})$. (one such is $I_{j_s}$). Consider for any $s$ the family of all those $I$ such that $I^\ast\subseteq I^\ast_{j_s}$. We write it as
\[
S'_{\phi,t,s}=\{I\in S'_{\phi,t}:I^\ast\subseteq I^\ast_{j_s}\}.
\]
For any $I\in S'_{\phi,t,s}$ we have of course that
\[
\frac{1}{\mi(I)}\int_I\phi d\mi>A, \qquad \text{so if we set} \qquad K_s=U\{I:I\in S'_{\phi,t,s}\}.
\]
We must have, because of the disjointness of the elements of family $S'_{\phi,t}$, that
\begin{eqnarray}
\frac{1}{\mi(K_s)}\int_{K_s}\phi d\mi>A.  \label{eq3.6}
\end{eqnarray}
Additionally, $K_s\subseteq I^\ast_{j_s}$ and by (\ref{eq3.5}) and the comments stated above we easily see that
\begin{eqnarray}
\frac{1}{k}\mi(I^\ast_{j_s})\le\mi(K_s)<\mi(I^\ast_{j_s}),  \label{eq3.7}
\end{eqnarray}
By (\ref{eq3.5}) and (\ref{eq3.6}) we can now choose (because $\mi$ is non-atomic) for any $s$, a measurable set $B_s\subseteq I^\ast_{j_s}\setminus K_s$, such that if we define $\Ga_s=K_s\cup B_s$, then $\dfrac{1}{\mi(\Ga_s)}\dis\int_{\Ga_s}\phi d\mi=A$.

We set now $E^\ast_t=\dis\bigcup_s I^\ast_{j_s}$
\[
\Ga=\bigcup_s\Ga_s, \ \ \De=\bigcup_s\De_s,
\]
where $\De_s=I^\ast_{j_s}\setminus\Ga_s$, for any $s=1,2,\ld\;.$

Then by all the above, we have that
\[
\Ga\cup\De=E^\ast_t \qquad \text{and} \qquad \frac{1}{\mi(\Ga)}\int_\Ga\phi d\mi=A_t,
\]
which is true in view of the pairwise disjointness of $\big(I^\ast_{j_s}\big)^\infty_{s=1}$.

Define now the following function
\[
h:=(\phi/\Ga)^\ast:(0,\mi(\Ga)]\ra\R^+.
\]
Then obviously
\[
\frac{1}{\mi(\Ga)}\int^{\mi(\Ga)}_0h(u)du=\frac{1}{\mi(\Ga)}\int_\Ga\phi d\mi=A_t.
\]
By the definition of $h$ we have that $h(u)\le\phi^\ast(u)$, for any $u\in(0,\mi(\Ga)]$. Thus we conclude:
\begin{eqnarray}
\frac{1}{\mi(\Ga)}\int^{\mi(\Ga)}_0\phi^\ast(u)du\ge\frac{1}{\mi(\Ga)}\int^{\mi(\Ga)}_0
h(u)du=A_t=\frac{1}{t}\int^t_0\phi^\ast(u)du,  \label{eq3.8}
\end{eqnarray}
From (\ref{eq3.8}), we have that $\mi(\Ga)\le t$, since $\phi^\ast$ is non-increasing.

We now consider a set $E\subseteq X$ such that $(\phi/E)^\ast=\phi^\ast/(0,t]$, with $\mi(E)=t$ and for which $\{\phi>\phi^\ast(t)\}\subseteq E\subseteq\{\phi\ge\phi^\ast(t)\}$.

It's existence is guaranteed by the equimeasurability of $\phi$ and $\phi^\ast$, and the fact that $(X,\mi)$ is non-atomic. Then, we see immediately that
\[
\frac{1}{\mi(E)}\int_E\phi d\mi=\frac{1}{t}\int^t_0\phi^\ast(u)du=A_t.
\]
We are going now to construct a second set $\widehat{E}\subseteq X$. We first set $\widehat{E}_1=\Ga$.

Let now $x\notin\widehat{E}_1$. Since $\Ga\supseteq\{\cm_\ct\phi>A_t\}$, we must have that $\cm_\ct\phi(x)\le A_t$. But since $\ct$ differentiates $L^1(X,\mi)$ we obviously have that for $\mi$-almost every $y\in X:\phi(y)\le\cm_\ct\phi(y)$. Then the set $\OO=\{x\notin\widehat{E}_1:\phi(x)>\cm_\ct\phi(x)\}$ has $\mi$-measure zero.

At last we set $\widehat{E}=\widehat{E}_1\cup\OO=\Ga\cup\OO$.

Then $\mi(\widehat{E})=\mi(\Ga)$ and for every $x\notin\widehat{E}$ we have that $\phi(x)\le\cm_\ct\phi(x)\le A_t$.

Let now $x\notin E$. By the construction of $E$ we immediately see that $\phi(x)\le\phi^\ast(t)\le\dfrac{1}{t}\dis\int^t_0\phi^\ast(u)du=A_t$. Thus, if $x\notin E$ or $x\notin\widehat{E}$, we must have that $\phi(x)\le A_t$, that is (\ref{eq2.3}) of Lemma \ref{lem2.1} is satisfied for these choices of $E$ and $\widehat{E}$. Let now $x\in\widehat{E}\setminus E$ and $y\in E$. Then we obviously have by the above discussion that $\phi(x)\le\phi^\ast(t)\le\phi(y)$. That is $\phi(x)\le\phi(y)$. Thus (\ref{eq2.4}) is also satisfied. Also since $\widehat{E}=\Ga\cup\OO$, we obviously have $\dfrac{1}{\mi(\widehat{E})}\dis\int_{\widehat{E}}\phi d\mi=A_t$, so as a consequence (\ref{eq2.2}) is satisfied also.

Applying Lemma \ref{lem2.1}, we conclude that
\[
\frac{1}{\mi(E)}\int_E\phi^pd\mi\le\frac{1}{\mi(\widehat{E})}\int_{\widehat{E}}\phi^pd\mi,
\]
or by the definitions of $E$ and $\widehat{E}$ that
\begin{eqnarray}
\frac{1}{t}\int^t_0[\phi^\ast(u)]^pdu\le\frac{1}{\mi(\Ga)}\int_\Ga\phi^pd\mi, \label{eq3.9}
\end{eqnarray}
Our aim is now to show that the right integral average in (\ref{eq3.9}) is less or equal that $(kc-k+1)(A_t)^p$. We proceed to this as follows:

We set $\el_\Ga=\dfrac{1}{\mi(\Ga)}\dis\int_\Ga\phi^pd\mi$. Then by the notation given above, we have that:
\begin{align}
\el_\Ga&=\frac{1}{\mi(\Ga)}\bigg(\int_{E^\ast_t}\phi^pd\mi-\int_\De\phi^pd\mi\bigg) \nonumber\\
&=\frac{1}{\mi(\Ga)}\bigg(\sum^\infty_{s=1}\int_{I^\ast_{j_s}}\phi^pd\mi-
\sum^\infty_{s=1}\int_{\De_s}\phi^pd\mi\bigg) \nonumber\\
&=\frac{1}{\mi(\Ga)}\sum^\infty_{s=1}p_s,  \label{eq3.10}
\end{align}
where the $p_s$ are given by
\[
p_s=\int_{I^\ast_{j_s}}\phi^pd\mi-\int_{\De_s}\phi^pd\mi, \qquad \text{for any} \qquad s=1,2,\ld\;.
\]
We find now an effective lower bound for the quantity $\dis\int_{\De_s}\phi^pd\mi$. By H\"{o}lder's inequality:
\begin{eqnarray}
\int_{\De_s}\phi^pd\mi\ge\frac{1}{\mi(\De_s)^{p-1}}\bigg(\int_{\De_s}\phi d\mi\bigg)^p,  \label{eq3.11}
\end{eqnarray}
Since $\De_s=I^\ast_{j_s}\setminus\Ga_s$, (\ref{eq3.11}) can be written as
\begin{eqnarray}
\int_{\De_s}\phi^pd\mi\ge\frac{\Big(\dis\int_{I^\ast_{j_s}}\phi d\mi-\dis\int_{\Ga_s}\phi du\Big)^p}{(\mi(I^\ast_{j_s})-\mi(\Ga_s))^{p-1}},  \label{eq3.12}
\end{eqnarray}
We now use H\"{o}lder's inequality in the form
\[
\frac{(\la_1+\la_2)^p}{(\si_1+\si_2)^{p-1}}\le\frac{\la_1^p}{\si^{p-1}_1}+\frac{\la^p_2}
{\si_2^{p-1}}, \quad \text{for} \quad \la_i\ge0 \quad \text{and} \quad \si_i>0
\]
which holds since $p>1$. Thus (\ref{eq3.12}) gives
\begin{eqnarray}
\int_{\De_s}\phi^pd\mi\ge\frac{1}{\mi(I^\ast_{j_s})^{p-1}}\bigg(\int_{I^\ast_{j_s}}
\phi d\mi\bigg)^p-\frac{1}{\mi(\Ga_s)^{p-1}}\bigg(\int_{\Ga_s}\phi d\mi\bigg)^p.  \label{eq3.13}
\end{eqnarray}
Since
$\dfrac{1}{\mi(\Ga_s)}\dis\int_{\Ga_s}\phi d\mi=A_t$,  (\ref{eq3.13}) gives
\[
\int_{\De_s}\phi^pd\mi\ge\frac{1}{\mi(I^\ast_{j_s})^{p-1}}\bigg(\int_{I^\ast_{j_s}}\phi d\mi\bigg)^p-\mi(\Ga_s)\cdot(A_t)^p,
\]
so we conclude, by the definition of $p_s$, that
\begin{eqnarray}
p_s\le\int_{I^\ast_{j_s}}\phi^pd\mi-\frac{1}{\mi(I^\ast_{j_s})^{p-1}}\bigg(\int_{I^\ast_{j_s}}
\phi d\mi\bigg)^p+\mi(\Ga_s)\cdot(A_t)^p, \label{eq3.14}
\end{eqnarray}
Using now (\ref{eq3.1}) for $I=I^\ast_{j_s}$, $s=1,2,\ld$ we have as a consequence that:
\begin{eqnarray}
p_s\le(c-1)\frac{1}{\mi(I^\ast_{j_s})^{p-1}}\bigg(\int_{I^\ast_{j_s}}\phi d\mi\bigg)^p+\mi(\Ga_s)(A_t)^p.  \label{eq3.15}
\end{eqnarray}
Summing now (\ref{eq3.15}) for $s=1,2,\ld$ we obtain in view of (\ref{eq3.10}) that
\begin{eqnarray}
\hspace*{1cm}\el_\Ga\le\frac{1}{\mi(\Ga)}\bigg[(c-1)\sum^\infty_{s=1}\frac{1}{\mi(I^\ast_{j_s})^{p-1}}\bigg(\int_{I^\ast_{j_s}}
\phi d\mi\bigg)^p+\bigg(\sum^\infty_{s=1}\mi(\Ga_s)\bigg)(A_t)^p\bigg].  \label{eq3.16}
\end{eqnarray}
Now from $\dfrac{1}{\mi(I^\ast_{j_s})}\dis\int_{I^\ast_{j_s}}\phi d\mi\le A_t$, we see that
\begin{align}
\el_\Ga&\le\frac{1}{\mi(\Ga)}\bigg[(c-1)\sum^\infty_{s=1}\mi(I^\ast_{j_s})\cdot(A_t)^p+\mi(\Ga)
\cdot(A_t)^p\bigg]  \nonumber \\
&=\bigg[(c-1)\frac{\mi(E^\ast_t)}{\mi(\Ga)}+1\bigg]\cdot(A_t)^p,  \label{eq3.17}
\end{align}
Since now $E^\ast_t\supseteq\Ga\supseteq E_t$, by (\ref{eq3.7}) we have that
\[
\mi(E^\ast_t)\le k\mi(E_t)\le K\mi(\Ga).
\]
Thus (\ref{eq3.17}) gives
\[
\frac{1}{\mi(\Ga)}\int_\Ga\phi^pd\mi\le[k(c-1)+1](A_t)^p.
\]
Using now (\ref{eq3.9}) and the last inequality we obtained the desired result. \hs
\begin{cor}\label{cor3.1}
If $\phi$ satisfies (\ref{eq3.1}) for every $I\in\ct$, then $\phi\in L^q$, for any $q\in[p,p_0)$, where $p_0$ is defined by $\dfrac{p_0-p}{p_0}\cdot\Big(\dfrac{p_0}{p_0-1}\Big)^p\cdot(kc-k+1)=1$.
\end{cor}
\begin{proof}
Immediate from Theorem 1 and A.  \hs
\end{proof}
\setcounter{rem}{0}
\begin{rem}\label{rem3.1}
All the above hold if we replace the condition (\ref{eq3.1}), by the known Muckenhoupt condition of $\phi$ over the dyadic sets of $X$. Then the same proof as above gives that the Muckenhoupt condition should hold for $\phi^\ast$, for the intervals of the form $(0,t]$, and for the constant $kc-k+1$. This is true since there exists analogous lemma as Lemma \ref{lem2.1} for this case (as can be seen in \cite{4}). Also the inequality that is used in order to produce (\ref{eq3.13}) from (\ref{eq3.12}) is true even for negative exponent $p<0$. We ommit the details. \hs
\end{rem}

Eleftherios N. Nikolidakis, Antonios D. Melas:
National and Kapodistrian University of Athens, Department of Mathematics, Zografou, GR-157 84
E-mail addresses: lefteris@math.uoc.gr, amelas@math.uoa.gr

\end{document}